\documentclass[a4paper,11pt]{amsart}

\pdfoutput=1

\usepackage{amsmath}
\usepackage{multicol}
\usepackage{mathrsfs}
\usepackage{amssymb}
\usepackage{amsthm}
\usepackage[top=95pt,bottom=80pt,left=85pt,right=85pt,footskip=35pt]{geometry}
\usepackage{xcolor}
\usepackage{enumerate}
\usepackage[all]{xy}
\xyoption{pdf}
\usepackage{tabularx}

\usepackage{makecell}
\usepackage{csquotes}
\usepackage{hyperref}
\usepackage{tikz}
\usetikzlibrary{shapes,positioning,intersections,quotes,calc}
\usepackage{ctable}
\usepackage{fancyhdr,floatpag}  

 \floatpagestyle{fancy}

\numberwithin{equation}{section}

\theoremstyle{definition}
\newtheorem*{Def}{Definition}
\newtheorem{Thm}{Theorem}[section]
\newtheorem{Prop}[Thm]{Proposition}
\newtheorem{Lem}[Thm]{Lemma}
\newtheorem{Rmk}[Thm]{Remark}
\newtheorem{Cor}[Thm]{Corollary}

\newcommand{\PP}{\mathbb{P}}
\newcommand{\CC}{\mathbb{C}}
\newcommand{\QQ}{\mathbb{Q}}
\newcommand{\ZZ}{\mathbb{Z}}

\newcommand{\Bl}{\mathrm{Bl}}
\newcommand{\Aut}{\operatorname{Aut}}

\newcommand{\Aa}{\mathbb{A}}
\newcommand{\cX}{\mathcal{X}}
\newcommand{\cO}{\mathcal{O}}
\newcommand{\cT}{\mathcal{T}}
\newcommand{\De}{{\Delta}}

\DeclareFontFamily{U}{mathc}{}
\DeclareFontShape{U}{mathc}{m}{it}%
{<->s*[1.03] mathc10}{}
\DeclareMathAlphabet{\mathcal}{U}{mathc}{m}{it}

\begin{document}

\title{Compact moduli of elliptic surfaces with a multiple fiber}
\author{Donggun Lee and Yongnam Lee}

\address{June E Huh Center for Mathematical Challenges\\
Korea Institute for Advanced Study\\
85 Hoegiro, Dongdaemun-gu\\
Seoul, 02455 Korea}
\email{dglee@kias.re.kr}

\address {Center for Complex Geometry\\
Institute for Basic Science (IBS)\\
55 Expo-ro, Yuseong-gu\\ 
Daejeon, 34126 Korea}
\email{ynlee@ibs.re.kr}

\thanks{MSC 2010: 14J10, 14J27\\ 
Key words: rational elliptic surface, Dolgachev surface, moduli space, $\QQ$-Gorenstein smoothing}
\date{September 7, 2026.}

\begin{abstract}
Motivated by Miranda and Ascher--Bejleri's works on compactifications of the moduli space of rational elliptic surfaces with a section, we study constructions and boundaries of compact moduli spaces of elliptic surfaces with a multiple fiber. Particular emphasis is placed on rational elliptic surfaces without a section and on Dolgachev surfaces. Our main goal is to understand the limit surfaces when a multiple fiber degenerates into an additive type singular fiber, via $\mathbb{Q}$-Gorenstein smoothings of slc surfaces.
\end{abstract}

\maketitle

\section{Introduction}

Constructions of elliptic surfaces starting from an elliptic surface with a section are well-established. Such procedures require either some knowledge of \'etale cohomology (Ogg--Shafarevich theory) or working with complex analytic surfaces (logarithmic transforms) which is not necessarily algebraic. For a discussion of the first method, see \cite[Chapter 4]{CDL} and for the second see \cite[Chapter 5]{BHPV}. Compactification of moduli of rational elliptic surfaces was established by Miranda using GIT \cite{Mir80} \cite{Mir81}. Also recently, Ascher and Bejleri \cite{AB21} used the log minimal model program to construct compact moduli spaces parameterizing weighted stable elliptic surfaces with a section and marked singular fibers each weighted between zero and one. 
Motivated by Miranda and Ascher--Bejleri's works on compactifications of moduli space of rational elliptic surfaces with a section, in this paper we try to establish some compactifications of moduli space of elliptic surfaces with a multiple fiber. To avoid some technical difficulties and to stay within the category of algebraic surfaces, we restrict to special elliptic surfaces whose Jacobian surfaces are rational and which have at most two multiple fibers.

We say that a smooth projective rational surface $X$ is a \emph{rational elliptic surface} if $X$ admits a relatively minimal fibration $p: X\to\PP^1$ whose generic fiber is an elliptic curve. If $X$ is a rational elliptic surface, then there exists an integer $m\ge 1$, called the \emph{index} of the fibration, such that $p$ is given by  $\lvert-mK_X\rvert$. Moreover, $m = 1$ if and only if $X\to\PP^1$ admits a global section and whenever $m > 1$ there exists exactly one multiple fiber in $X$, of multiplicity $m$. 

Rational elliptic surfaces of index $m$ can be realized as the blow-up of $\PP^2$ at nine points, where the nine points are base points of a Halphen pencil (of index $m$) \cite[Section~4.9]{CDL}; these are pencils of plane curves of degree $3m$ having nine (possibly infinitely near) base points of multiplicity $m$. The unique multiple fiber corresponds to the cubic curve through nine base points. Furthermore, any rational elliptic surface of index $m$ carries an $m$-multi-section, which can be chosen as the exceptional divisor of over the ninth point. We refer to such a pair a \emph{marked rational elliptic surface of index $m$}. 

A \emph{Dolgachev surface} is an elliptic surface of Kodaira dimension one with two multiple fibers of multiplicities $p$ and $q$ where $p$ and $q$ are relatively prime. Let $\pi: X\to\PP^1$ be rational elliptic surface of index 1 and let $(p, q)$ be a pair of natural numbers. We form a new algebraic surface $S(p,q)$ by performing logarithmic transforms on two of the smooth or multiplicative fibers of $\pi$, one of order $p$ and the other of order $q$. Then there is an elliptic fibration $\pi: S(p, q) \to \PP^1$ with two multiple fibers of multiplicities $p$ and $q$. Dolgachev \cite{Dol} showed that if $(p, q) = 1$, then $S(p, q)$ is simply connected. We shall use the term \emph{Dolgachev surface of the type $(p, q)$} to mean an algebraic surface of the form $S(p, q)$ where $p$ and $q$ are relatively prime and $p,q > 1$.

Zanardini \cite{Zar} studied the problem of classifying pencils of plane sextics up to projective equivalence via GIT and provided a complete description of the GIT stability of certain pencils of plane sextics called Halphen pencils of index two. This study gives a compactification of moduli of rational elliptic surfaces with a multiple fiber of multiplicity 2. However, as the multiplicity increases it becomes very difficult to describe semi-stable objects via GIT computation. Additionally, Dolgachev surfaces are constructed by the non-global method of performing logarithmic  transforms at two of the smooth or multiplicative fibers of rational elliptic surfaces with a section. Only very special Dolgachev surfaces of type $(2, 3)$ whose Jacobian surfaces are extremal rational elliptic surfaces have been explicitly constructed through a finite cover and a finite quotient method. In general, Dolgachev surfaces have not been explicitly described beyond such specific constructions, which makes it difficult to study their degenerations.

Let $\pi: X \to\PP^1$ be a rational elliptic surface, and let $x, y\in \PP^1$ be two points such that the fibers $\pi^{-1}(x) =F_x$ and $\pi^{-1}(y) =F_y$  are either smooth or singular fibers of multiplicative type. The construction of a Dolgachev surface $S(p, q)$ from $X$ depends on a choice of two line bundles $\eta_x\in {\rm Pic}(F_x), \eta_y\in {\rm Pic}(F_y)$ of orders exactly $p$ and $q$ respectively. Since ${\rm Br}(X)=0$ \cite[pp 123--126]{Dol}, there is a unique elliptic surface $S(p, q) \to\PP^1$ with multiple fibers at $x$ and $y$, with associated invariants $\eta_x$ and $\eta_y$, that is isomorphic to $X$ over $\PP^1\setminus \{x, y\}$. 

Let $T$ denote the space of quintuples: $T= \{(\pi:X\to \PP^1, x, y, \eta_x, \eta_y)\}$
subject to the conditions that $\pi: X \to\PP^1$ is a rational elliptic surface, 
$x\ne y$ are points of $\PP^1$, $F_x$ and $F_y$ are smooth fibers or singular fibers of multiplicative type of $\pi$, $\eta_x\in{\rm Pic}(F_x)$ is a $p$-torsion point and $\eta_y\in{\rm Pic}(F_y)$ is a $q$-torsion point.

\begin{Prop}\cite[Proposition~3.8]{FM88} For given relatively prime integers $p$ and $q$, the moduli space of Dolgachev surfaces with multiple fibers of multiplicities $p$ and $q$ is irreducible, i.e. there exists an irreducible complex space $T$ (which may be assumed smooth) and a proper smooth map $\Phi: \cX\to T$ such that every Dolgachev surface $S(p,q)$ is isomorphic to $\Phi^{-1}(t)$ for some $t\in T$.
\end{Prop}
Here $T$ is the same as before and the family $\Phi$ is the specific family constructed from that data, and by construction, every fiber of $\Phi$ is a Dolgachev surface $S(p,q)$.

Let $X$ be the Jacobian surface associated to $S(p, q)$. If $X$ has no additive singular fibers, then for any two distinct points $x\ne y$ one can construct $S(p, q)$ by performing logarithmic transforms at $F_x$ and $F_y$ by using their torsions. 
Therefore, when studying degenerations of Dolgachev surfaces, the main difficulty arises when a multiple fiber degenerates into an additive singular fiber. Our goal in this paper is to understand the limit surfaces that appear when a multiple fiber degenerates into an additive singular fiber in an elliptic surface with $p_g=q=0$.

\medskip

In Section~\ref{prel}, we briefly review the work of Miranda and Ascher--Bejleri on compactifications of the moduli spaces of elliptic surfaces with a section from our perspective. 

In Section~\ref{stable}, we construct two compactifications of the moduli space of marked rational elliptic surfaces of index $m$. First, using KSBA-stable pairs of the form
\[(X,F_m+\bar A+cF_{sing}),\]
where $F_m$ is the reduced support of the multiple fiber, $\bar A$ is the chosen multi-section, and $F_{sing}$ is the sum of singular fibers, we construct KSBA compactifications $\mathcal R_m^{\mathrm{KSBA}}(c)$ for $\frac{1}{12m}<c\leq 1$ (Theorem~\ref{thm:RES.KSBA}).
We also interpret marked rational elliptic surfaces of index $m$ as stable Calabi--Yau pairs in the sense of Birkar \cite{Bir} by considering
\[(X,F_m),A \qquad \text{with } A=\bar A+F_{sing},\]
which leads to the compactification $\mathcal R_m^{\mathrm{CY}}$ (Theorem~\ref{thm:RES.CY}).
Finally, we prove that every Dolgachev surface $S(p,q)$ admits a $pq$-multi-section (Proposition~\ref{multisection}), and discuss a possible KSBA approach to compactifying marked Dolgachev surfaces, conditional on the existence of flat deformations of such multi-sections in families.

In Section~\ref{moderate}, we review Kawamata's work \cite{Kaw} on the classification of possible singularities of the central fiber of a moderate (=permissible degeneration with a normal central fiber) degeneration of elliptic surfaces with non-negative Kodaira dimension. In a moderate degeneration, singular fiber types $mI_0, I_b^*, II^*, III^*, IV^*$ do not occur on the central fiber. For Dolgachev surfaces, the converse of the classification theorem is also true. Consider a relatively minimal smooth rational elliptic surface $Y$ with a section. Assume $Y$ has at least two singular fibers of type $I_n, II, III, IV$. Then one can construct a normal rational elliptic surface $X$ by replacing these two singular fibers with singular fibers in Kawamata's theorem \cite[Theorem~4.2]{Kaw} by blowing-ups and Artin's contractibility theorem \cite[Theorem~2.3]{Art}.  Then by the method in \cite[Section~2]{LP}, or \cite[Section~6]{LN}, or \cite[Section~2]{CL}, we obtain a $\QQ$-Gorenstein smoothing $f: \cX\to\De$ such that $f^{-1}(0)=X$ and a general fiber $X_t$ is a Dolgachev surface. Precisely, we prove the following theorem.
\begin{Thm}(=Theorem~\ref{smoothing})
Let $X$ be a normal projective surface with an elliptic fibration. Suppose that the minimal resolution of $X$ is a rational elliptic surface with a section and non-isotrivial elliptic fibration. Suppose there are at most two singular fibers $L$ whose $L_{\rm red}$ is one of the types $I_d(r, a), II(r), III(r)$, or $IV(r)$ where $r=2,3,4,5$ for type $II$, $r=2, 3$ for type $III$, and $r=2$ for type $IV$. Then there exists a $\QQ$-Gorenstein smoothing to elliptic surfaces. 

If there is exactly one such singular fiber $L_{\rm red}$, this $\QQ$-Gorenstein smoothing gives a rational elliptic surface of index $r$. If there are two such singular fibers $L_{\rm red}$, the $\QQ$-Gorenstein smoothing gives an elliptic surface with non-negative Kodaira dimension, unless both multiplicities are two, in which case it gives an Enriques surface.
\end{Thm}
The advantage of constructing Dolgachev surfaces with $\QQ$-Gorenstein smoothings is that it is extendable to arbitrary positive characteristic.

In Section~\ref{limit}, we describe the limit surfaces when a multiple fiber goes to an additive singular fiber in an elliptic surface with $p_g=q=0$. First, we study the existence of a $\QQ$-Gorenstein smoothing for the following semi-log terminal surfaces.
\begin{Thm}(=Theorem~\ref{gluing})
A semi-log terminal surface $X=X_1\cup_{\PP^1} X_2$ has a $\QQ$-Gorenstein smoothing to rational elliptic surfaces of index 1 if it is the union of two rational elliptic surfaces of index 1 with cyclic quotient singularities along twisted $I_0^*\, |\, I_0^*$,  $II\, |\, II^*$,  $III\, |\, III^*$, or $IV\, |\, IV^*$. 
\end{Thm}

The detailed statement is given in Theorem~\ref{gluing}. As a corollary, we understand the limit surfaces when a multiple fiber degenerates into an additive type singular fiber, via $\mathbb{Q}$-Gorenstein smoothings of semi-log canonical surfaces.
\begin{Cor}(=Corollary~\ref{gluing2})
Let $X=X_1\cup_{\PP^1} X_2$ be the semi-log terminal surface obtained as the union of two rational elliptic surfaces of indices $m_1$ and $m_2$ with $(m_1,m_2)=1$, whose gluing is one of the types in Theorem~\ref{gluing}.
Let $m_iF_i$ be the multiple fiber of $X_i$.
Let $Y$ be the associated Jacobian surface.
Assume that $F_i$ are smooth, and that the moduli map $\PP^1\cup\PP^1\dasharrow \mathcal M_{1,1}$ induced by the elliptic fibration on $Y$ is \'etale at the corresponding points.
Then $X$ has a $\QQ$-Gorenstein smoothing to Dolgachev surfaces of type $(m_1,m_2)$.
\end{Cor}

\noindent{\bf Acknowledgements.}
This research was partially supported by the Institute for Basic Science (IBS-R032-D1) and the KIAS Individual Grant (HP109201). The authors would like to thank Igor Dolgachev, Keiji Oguiso, and Guolei Zhong for their valuable comments during this work. The authors also thank Marcos Canedo and Giancarlo Urz\'ua for explaining their very recent paper, and the anonymous reviewers for providing valuable feedback that helped to improve this paper.

\section{Preliminaries}\label{prel}

\noindent{\bf Terminology and Notation.}
We adopt the standard terminology and notation as in \cite{Har}. We refer to \cite{KM} for the definitions of singularities, and to \cite[Chapter V, Section 7]{BHPV} for Kodaira's table of singular fibers.

\medskip

We now briefly introduce Miranda's work \cite{Mir80} \cite{Mir81}, and Ascher and Bejleri's work \cite{AB21} \cite{AB22} from our perspective.

Let $V$ be the vector space of sections $\Gamma(\PP^2, \cO_{\PP^2}(1))$. Then the projective space $\PP^9=\PP(S^3 V^*)$ is the parameter space for cubic curves in $\PP^2$. Let ${\rm Gr}(1, 9)$ be the Grassmannian  of lines in this $\PP^9$; a point of ${\rm Gr}(1, 9)$ then corresponds to a pencil of plane cubic curves.  ${\rm Gr}(1, 9)$ is naturally embedded in the projective space
$\PP^{44}= \PP(\Lambda^2S^3 V^*)$ via the Pl\"ucker coordinates. The automorphism group $PGL(3)$ of $\PP^2$ acts naturally on the space ${\rm Gr}(1, 9)$ of all cubic pencils. In \cite{Mir80}, Miranda gave an explicit geometric characterization of the GIT stability of pencils with smooth members by considering the elliptic surface $X_P$ associated to such a pencil $P$ ($X_P$ is obtained by blowing up $\PP^2$ at the base points of the pencil $P$).

\begin{Thm}\cite{Mir80}\label{Mir80-1}
Let $P$ be a pencil of cubic curves in $\PP^2$. 
\begin{enumerate}
    \item $P$ is stable if and only if $P$ contains a smooth member and every fiber of $X_P$ is reduced.
    \item If $P$ contains a smooth member, then $P$ is semi-stable if and only if $X_P$ contains no fibers of type $II^*, III^*$, or $IV^*$.
\end{enumerate}
\end{Thm}

He also characterized strictly semi-stable pencils containing a smooth member. 

\begin{Thm}\cite[Theorems~8.1 and~8.5]{Mir80}\label{Mir80-2}
Suppose that $P$ 
contains a smooth member.
\begin{enumerate}
    \item $P$ is strictly semi-stable if and only if $X_P$ contains a fiber of type $I_n^*$.
    \item Let $P$ be strictly semi-stable. Then, the orbit of $P$ is closed if and only if $X_P$ contains two singular fibers of type $I_0^*$.
\end{enumerate}
\end{Thm}

Miranda  \cite{Mir81} also used GIT to construct a coarse moduli space of Weierstrass fibrations. These fibrations arise naturally as follows: let $p: X\to C$ be a rational elliptic surface 
with a section $S$. One obtains a normal surface, called the Weierstrass fibration associated to $X\to C$ by contracting all the components of the fibers of $p$ which do not meet $S$. This fibration has only rational double point singularities, and is uniquely determined by $X$. 

Let $\Gamma_n = \Gamma(\PP^1, \cO_{\PP^1}(n))$. The Weierstrass fibration of a relatively minimal elliptic surface $X\to\PP^1$ is the closed subscheme of $\PP(\cO_{\PP^1}(2N) \oplus\cO_{\PP^1}(3N) \oplus\cO_{\PP^1})$ defined by the equation $y^2z=x^3+Axz^2+Bz^3$, where $A\in \Gamma_{4N}$, $B \in\Gamma_{6N}$,  $N={\rm deg}\, (p_*\omega_{X/\PP^1})$, and
\begin{enumerate}
    \item[(i)] $4A(q)^3 + 27B(q)^2 = 0$ precisely at the (finitely many) singular fibers $X_q$, and
    \item[(ii)] for each $q\in \PP^1$ we have $\nu_q(A)\le 3$, or $\nu_q(B)\le 5$.
\end{enumerate}

Let $T\subset \Gamma_{4N}\oplus \Gamma_{6N}$ be the open set of pairs $(A,B)$ satisfying (i) and (ii) above. Given a Weierstrass fibration $X\to\PP^1$, the pair $(A, B)$ is unique up to isomorphism in the following sense. The multiplicative group $\CC^*$ acts on $T$ by $\lambda\cdot (A, B)= (\lambda^{4N}A, \lambda^{6N}B)$. The group $SL(2)$ acts on $T$ in the obvious manner because $\Gamma_n=S^n\Gamma_1$. 

\smallskip

Let $W$ denote the GIT quotient, and let $W_{sss}$ denote the strictly semistable locus.

\begin{Thm}\cite[Theorem~6.2, Proposition~8.2, Theorem~8.3]{Mir81}
Let $X$ be a rational Weierstrass fibration (i.e. $N=1$), and let $\tilde X$ be the associated elliptic surface. 
Then,
\begin{enumerate}
    \item $X$ is stable if and only if $X$ has a smooth generic fiber and 
    $\tilde X$ has only reduced fibers;
    \item $X$ is strictly semi-stable if and only if $\tilde X$ has a fiber of type $I_n^*$ for some $n\ge 0$;
    \item two strictly semi-stable elliptic surfaces correspond to the same point in $W_{sss}$ if and only if the $j$-invariant of the $I_n^*$ fibers are the same.
\end{enumerate}
\end{Thm}
In particular, there is a stratification $W=W_s \sqcup \Aa^1 \sqcup \infty$ where $W_s$ denotes the stable locus, $\Aa^1$ parametrizes elliptic surfaces with an $I_0^*$-fiber via its $j$-invariant, and $\infty$ corresponds to elliptic surfaces with an $I_n^*$-fiber for $n\geq1$. 
These results determine the singular fibers that occur on elliptic surfaces corresponding to points in the stable and strictly semistable loci of the GIT quotient.

Recently, Ascher and Bejleri \cite{AB21} used the log minimal model program to construct compact moduli spaces parameterizing weighted stable elliptic surfaces which have elliptic fibrations with a section and marked fibers, each assigned a weight between zero and one. Moreover, they showed that the domain of weights admits a wall-and-chamber structure, described the induced wall-crossing morphisms on the moduli spaces as the weight vector varies, and identified the surfaces that appear on the boundary of the moduli space. In particular, when $N=1$, Ascher and Bejleri \cite{AB22} constructed a stable pair compactification of the moduli space of anti-canonically polarized degree one del Pezzo surfaces and related their constructions to Miranda’s GIT quotient \cite{Mir81}.

The blow-up of a degree one del Pezzo surface $\bar X$ at the base point of $\lvert-K_{\bar X} \rvert$ is a rational elliptic surface $X\to\PP^1$ with a section given by the exceptional divisor. Equivalently, $\bar X$ may be obtained as the blow-up of $\PP^2$ at 8 points in general position, and the anticanonical pencil is the unique pencil of cubics passing through these points. By the Cayley--Bacharach theorem, there is a unique 9-th point in the base locus of this pencil which becomes the base point of 
$\lvert-K_{\bar X}\rvert$.

\begin{Thm}\cite{AB22}
There exists a proper Deligne--Mumford stack $R$ parametrizing anticanonically polarized broken del Pezzo surfaces of degree one with the following properties:
\begin{itemize}
\item The interior $U\subset R$ parametrizes degree one del Pezzo surfaces with at worst rational double points (RDPs).
\item The complement $R \setminus U$ is a divisor consisting of 
2-Gorenstein semi-log-canonical (slc) surfaces with ample anticanonical divisor and exactly two irreducible components, or of isotrivial surfaces with $j$-invariant equal to infinity.
\item The locus $R^\circ\subset R$ parametrizing surfaces such that every irreducible component is normal, is a smooth Deligne--Mumford stack.
\end{itemize}
\end{Thm}

\begin{Thm}\cite[Theorem~4.5]{AB22}
The surfaces parametrized by $R^\circ$ are either: 
\begin{enumerate}
    \item[(i)] normal degree one del Pezzo surfaces with RDPs, whose singular pseudofibers\footnote{Pseudofibers are the images of fibers of an elliptic surface with a section under the contraction of the section.} are Weierstrass of type $I_n, II, III$ or $IV$; or 
    \item[(ii)] semi-log-canonical (slc) unions of two degree one del Pezzo surfaces with, glued along twisted $I_0^*$ fibers of index 2, with all other singular fibers as in (i).
\end{enumerate}
\end{Thm}

\section{Moduli of rational elliptic surfaces of index $m$ and Dolgachev surfaces}\label{stable}

Let $X$ be a rational elliptic surface of index $m\ge 2$ and let $F_o=mF_m$ be the unique multiple fiber of $p: X\to\PP^1$. Let $\bar A$ be an $m$-multi-section, arising as one of the nine exceptional divisors when viewing $X$ as the blow-up of $\PP^2$ at nine points.
By definition, $X$ together with $\bar A$ is a marked rational elliptic surface of index $m$.
Moreover,
\[K_X+F_m\sim 0.\]

Suppose all fibers are irreducible, the reduced support $F_m$ is smooth, and all singular fibers are nodal. In particular, there are precisely $12$ singular fibers. Let $F_{sing}$ denote their sum.
Using $\bar A$ and $F_{sing}$, we obtain explicit ample divisors $\bar A+cF_{sing}$ for $\frac{1}{12m}<c\leq 1$. These are ample by the Nakai--Moishezon criterion and $\bar A^2=-1$. 
Hence
\[K_X+F_m+\bar A+cF_{sing}\sim\bar A+cF_{sing}\]
is ample. Since the pair $(X,F_m+\bar A+cF_{sing})$ is slc, it is KSBA-stable \cite{Kol23}.
These pairs may be viewed as analogues of those considered in \cite{AB17,AB21,AB22} for rational elliptic surfaces with a section.

Consider the locally closed subvariety $U_m\subset (\PP^2)^9$
consisting of $(p_i)_{i=1,\cdots,9}$ such that 
\begin{itemize}
    \item $X=\Bl_{p_1,\cdots,p_9}\PP^2$ is a rational elliptic surface of index $m$ with irreducible fibers\\ (this is called an \emph{unnodal} Halphen surface of index $m$ in \cite[Section~2]{CD12}), 
    \item the support $F_m$ of its unique multiple fiber $F_o=mF_m$ is smooth, 
    \item the ninth exceptional divisor $\bar A$ meets fibers of $X\to \PP^1$ in their smooth loci,
    \item all singular fibers are nodal; let $F_{sing}$ denote their sum.
\end{itemize}

\begin{Thm}\label{thm:RES.KSBA}
    Let $m\geq2$, and let $c\in\mathbb Q, \ \frac{1}{12m}<c\leq 1$. 
    Then there exists a proper Deligne--Mumford stack $\mathcal R_m^{\mathrm{KSBA}}(c)$ with projective coarse moduli space $R_m^{\mathrm{KSBA}}(c)$, whose general point corresponds to a marked rational elliptic surface of index $m$ arising from $U_m$. 
\end{Thm}
\begin{proof}
	Let $\cX\to U_m$ be the family of rational elliptic surfaces of index $m$ obtained via blowing up the family of nine points in the trivial $\PP^2$-bundle over $U_m$, and let $\mathcal F_m$ be the family over $U_m$ of the reduced supports of the multiple fibers in $\cX$.
	Let $\bar{\mathcal A}$ be the exceptional divisor over the ninth point. Let $\mathcal F_{sing}$ be the family of unions of singular fibers. 
	Then the tuples $(\cX;\mathcal F_m,\bar{\mathcal A},c\mathcal F_{sing})\to U_m$
    define a family of KSBA-stable pairs, and hence induces a morphism from $U_m$ to the moduli stack of KSBA-stable pairs \cite{Kol23}. 
	The closure of the image of $U_m$ provides the desired proper Deligne--Mumford stack.
	
	Note that two marked rational elliptic surfaces of index $m$ parametrized by $U_m$ are isomorphic if and only if the corresponding tuples $(X;F_m,\bar A,cF_{sing})$ are isomorphic.
\end{proof}

Marked rational elliptic surfaces of index $m$ also fit naturally into the framework of stable Calabi--Yau pairs due to Birkar \cite[Definitions~1.1,~1.8 and~1.13]{Bir}.

\begin{Def}
A \emph{stable minimal model} $(X , B), A$ consists of a connected projective pair $(X , B)$ and a $\QQ$-Cartier divisor $A\ge 0$ such that: 
\begin{itemize}
\item $(X, B)$ is slc, 
\item $K_X+B$ is semi-ample defining a contraction $f: X\to Z$,
\item $K_X+B+tA$ is ample for some $t >0$,
\item $(X, B+tA)$ is slc for some $t >0$.
\end{itemize}
When $K_X+B\sim_\QQ 0$, we call it a \emph{stable Calabi--Yau pair}.
\end{Def}

\begin{Def}
    Fix $d\in \mathbb N$, $c,v \in \QQ_{>0}$, and $\sigma \in \QQ[t]$. 
    A \emph{$(d,c,v, \sigma)$-stable minimal model} is a stable minimal model $(X,B), A$ such that:
    \begin{itemize}
        \item $\dim X=d$,
        \item the coefficients of $A$ and $B$ are in $c\mathbb Z_{\geq0}$,
        \item $\mathrm{vol}(A|_F)= v$, where $F$ is any general fiber of the fibration $f:X\to Z$ determined by $K_X+B$ over any irreducible component of $Z$,
        \item $\mathrm{vol}(K_X+B+tA)=\sigma (t)$ for $0\leq t \ll 1$.
    \end{itemize}
    When $K_X+B\sim_\QQ 0$, we call it a \emph{$(d,c,v)$-stable Calabi--Yau pair}.
\end{Def}
\begin{Def}
    Let $U$ be a reduced scheme over $\CC$. 
    A \emph{family of $(d,c,v,\sigma)$-stable minimal models over $U$} consists of a projective morphism $X\to U$ of schemes and $\QQ$-divisors $B$ and $A$ on $X$ such that:
    \begin{itemize}
        \item $(X,B+tA)\to U$ is a locally stable family for every sufficiently small $t\in\QQ_{\geq0}$,
        \item $B=cD$ and $A=cN$ for relative Mumford divisors $D,N\geq 0$,
        \item $(X_u,B_u), A_u$ is a $(d,c,v,\sigma)$-stable minimal model over $k(u)$ for each $u\in U$.
    \end{itemize}
    When $K_{X/U}+B\sim_\QQ 0/U$, we call it a \emph{family of $(d,c,v)$-stable Calabi--Yau pairs}.
\end{Def}
Here, $B_u,A_u$ are the divisorial pullbacks of $B,A$ to $X_u$, 
respectively \cite[Definition~4.6]{Kol23}. For a definition of a locally stable family, see \cite[Definition--Theorem~4.7]{Kol23}. 

\begin{Thm} \cite[Theorem~1.14]{Bir} \label{t.bir2}
     Fix $d\in \mathbb N$, $c,v\in \QQ_{>0}$, and $\sigma \in \QQ[t]$. Then, there exists a proper Deligne--Mumford stack $\mathcal M_{d,c,v,\sigma}$ over $\CC$ 
     such that $\mathcal M_{d,c,v,\sigma}(U)$ is isomorphic to the groupoid of families of $(d,c,v,\sigma)$-stable minimal models over $U$ as groupoids for every reduced scheme $U$ over $\CC$. 
     Moreover, it admits a projective coarse moduli space $M_{d,c,v,\sigma}$. .
\end{Thm}

We write $\mathcal P_{d,c,v}:=\mathcal M_{d,c,v,\sigma}$ and $P_{d,c,v}:=M_{d,c,v,\sigma}$ in the case of Calabi--Yau pairs.

\medskip

It is straightforward to check that for every $(X,F_m)$ parametrized by $U_m$, the pair 
\[(X,F_m),A \quad \text{ with }~ A:=\bar A+F_{sing}\]
is a $(2,1,24m-1)$-stable Calabi--Yau pair. 
Indeed, $K_X+F_m+tA\sim_\QQ tA$ is ample for any $t>0$, 
and we have $A^2=24m-1$.
Thus the moduli theory of stable Calabi--Yau pairs also provides a projective moduli space of marked rational elliptic surfaces of index $m$.
\begin{Thm}\label{thm:RES.CY}
    Let $m\geq2$. 
    Then there exists a proper Deligne--Mumford stack $\mathcal R_m^{\mathrm{CY}}\subset \mathcal P_{2,1,24m-1}$ with projective coarse moduli space $R_m^{\mathrm{CY}}$, whose general point corresponds to a marked rational elliptic surface of index $m$ arising from $U_m$. 
\end{Thm}
\begin{proof}
    Let $\cX, \mathcal F_m, \bar{\mathcal A}$, and $\mathcal F_{sing}$ be as in the proof of Theorem~\ref{thm:RES.KSBA}. 
    Then, $(\cX,\mathcal F_m)\to U_m$ together with $\mathcal A:=\bar{\mathcal A}+\mathcal F_{sing}$ is a family of $(2,1,24m-1)$-stable Calabi--Yau pairs. 
    By the universal property of $\mathcal P_{2,1,24m-1}$, this induces a morphism $U_m\to \mathcal P_{2,1,24m-1}$, and 
    the closure of the image of $U_m$ provides the desired proper Deligne--Mumford stack.
    
    Note that two marked rational elliptic surfaces of index $m$ parametrized by $U_m$ are isomorphic if and only if the corresponding $(X,F_m),A$ are isomorphic.
\end{proof}

\begin{Rmk}
For both compactifications $\mathcal R_m^{\mathrm{KSBA}}(c)$ and $\mathcal R_m^{\mathrm{CY}}$, the generic stabilizer is trivial for $m>2$, and isomorphic to $\ZZ/2\ZZ$ for $m=2$.

Indeed, this follows from the description of $\Aut(X)$ for general $X$ in the last paragraph of \cite[Remark~2.11]{CD12}, since every automorphism preserves the elliptic fibration and its singular fibers, while a nontrivial fiberwise translation does not preserve the chosen $m$-multi-section.
\end{Rmk}

The preceding construction suggests a similar approach for Dolgachev surfaces. 
A first step is the existence of suitable multi-sections, established in the following proposition.

\begin{Prop}\label{multisection}
Let \(S:=S(p, q)\) be a Dolgachev surface, i.e., there is  a minimal elliptic fibration \(f: S\to B\cong\mathbb{P}^1\) with two multiple fibres \(F_P\) and \(F_Q\) of coprime multiplicities \(p\) and \(q\). 
Then there is a \(pq\)-multi-section. 
\end{Prop}

\begin{proof}
The existence of a $pq$-multi-section is obtained by the Ogg--Shafarevich theory (cf. \cite[Chapter 4]{CDL}). Since the Brauer group of the Jacobian surface $J(S)$ of $S$ is zero, the order of the torsor $x\in H^1_{\text{\'et}}(\eta, J(S)_\eta)$ for the generic point $\eta\in B$ corresponding to the elliptic fibration $f: S\to B$ coincides with the smallest degree of a multi-section of $f$ \cite[Proposition~4.6.5 and Corollary 4.6.6]{CDL}. \end{proof}

\begin{Rmk}
Proposition~\ref{multisection} shows the existence of a \(pq\)-multi-section, although we do not know whether it is smooth. We expect that there exists a \(pq\)-multi-section $C$ which is isomorphic to $\PP^1$, in which case the canonical bundle formula (cf.~\cite[Chapter~2]{Dol}) shows that $C^2=p+q-pq-2$.
It is true for known examples.
\end{Rmk}

At present, we do not know whether such a multi-section admits a flat deformation in families of Dolgachev surfaces.
Nevertheless, one may try to construct KSBA compactifications of Dolgachev surfaces by considering pairs of the form
\[(S,c(F_x+F_y)+A+c'F_{sing}),\]
where $A$ is a multi-section, $0\leq c,c'\leq1$, $F_x=(F_P)_{\rm red}$, and $F_y=(F_Q)_{\rm red}$.
If the multi-section admits a flat deformation in a family of Dolgachev surfaces and the resulting pairs form a family of KSBA-stable pairs, then one obtains a morphism from the base to the corresponding KSBA moduli stack. Taking the closure of the image provides a compact moduli space of marked Dolgachev surfaces, while understanding the resulting boundary points requires further investigation.
The precise range of admissible coefficients depends on the geometry of the multi-section $A$, in particular on its self-intersection $A^2$.

\section{Moderate degeneration of elliptic surfaces with non-negative Kodaira dimension}\label{moderate}

Let $f^o: \cX^o\to\De^o=\De-\{ 0\}$ be a flat family of smooth minimal projective surfaces with non-negative Kodaira dimension. Then by the theory of semi-stable minimal models of threefolds \cite[Section~7]{KM}, there exists a $\QQ$-factorial terminal threefold $\cX$ together with a flat morphism $f:\cX\to \De$ compactifying $f^o$, with $K_{\cX}$ $\QQ$-nef. The central fiber $f^{-1}(0)$ is then a projective surface with semi-log-terminal (slt) singularities.

If the central fiber $X_0: = f^{-1}(0)$ is normal, then RDPs or quotient singularities of type $\frac{1}{r^2d}(1, ard-1)$ with $(a, r)=1$ can occur in $X_0$. By a partial resolution, a singularity of type $\frac{1}{r^2d}(1, ard-1)$ splits into $d$-number of singularities of type $\frac{1}{r^2}(1, ar-1)$. Additionally by the base change and analytic $\QQ$-factorialization, Kawamata \cite[Theorem~1.3]{Kaw} showed that there is an {\it analytically $\QQ$-factorial}  terminal threefold $\cX$ over a disk $\De$ which compactify $f^o$, and $f: \cX\to\De$ is a minimal permissible degeneration of projective surfaces. We refer to \cite[Definition 1.1]{Kaw} for the definition of permissible degeneration. Permissible degeneration is obtained after some analytic modification of semi-stable minimal model of threefolds. Then some semi-log-terminal singularites are excluded. For instance, there is no pinch point. RDPs are also excluded through base change and simultaneous resolution \cite[Proof of Theorem 1.3]{Kaw}.

A {\it moderate degeneration} 
is a permissible degeneration $f:\cX \to \De$ such that the central fiber $f^{-1}(0)$ is normal. If $f$ is a moderate degeneration then we have the following.
\begin{enumerate}
\item The surface $X_0:=f^{-1}(0)$ has at worst cyclic quotient singularities of type $1/{r^2}(a, r-a)$ with $(a, r)=1$ which is not RDP.
\item For $t\in \De^o$, $K^2_{X_0}=K^2_{X_t}$ and $\chi(X_0, \cO_{X_0})= \chi(X_t, \cO_{X_t})$.
\item $e(X_0)=e(X_t)$, and the Noether formula holds.
\item If $K_\cX$ is $f$-nef and $f$ is not smooth, then there exist positive integers $m_1$ and $m_2$ such that $$h^0(X_t, mK_{X_t}) > h^0(S, mK_S)$$ for all $m$ with $m_1| m$ and $m> m_2$ where $S$ is the minimal resolution of $X_0$.
\end{enumerate}

The possible singularities and the structure of the central fiber in moderate degeneration are completely classified by Kawamata when $\kappa(X_t)$ is 0 or 1.
\begin{Thm} \cite[Theorem~4.1]{Kaw}. Let $f: \cX\to\De$  be a moderate degeneration of surfaces with 
$\kappa(X_t)=0$ and $f$ 
not smooth. Then $X_t$ is an Enriques surface, and $X_0$ is a singular rational surface with only quotient singularities of type $1/4(1, 1)$.
\end{Thm}

\begin{Thm}\label{classification} \cite[Theorem~4.2]{Kaw}. Let $f: \cX\to\De$  be a moderate degeneration of surfaces with 
$\kappa(X_t)=1$ and $f$ 
not smooth. Then there exists a smooth surface $B$, a proper surjective morphism $g: \cX\to B$ with general fibers elliptic curves, and a proper smooth morphism $h: B\to\De$ such that $f=h\circ g$. In particular, the central fiber also admits an elliptic fibration.
\end{Thm}

The singular fibers below are classified in \cite[Theorem~4.2]{Kaw}.
Let  $\sigma: S\to X_0$ be the minimal resolution, and let $g_{X_0}: X_0\to B_0$ and $g_S: S\to B_0$ be the elliptic fibrations 
induced by $g$. Let $\tau: S\to \bar S$ be the contraction to the relative minimal model $g_{\bar S}: \bar S\to B_0$. 

Let $L$ be a scheme theoretic fiber of $g_{X_0}$ which passes through some singular points of $\cX$. Let $L_S=\sigma^*L$ and  $L_{\bar S} =\tau_*L_S$. Let $\tilde m$ be the greatest common divisor of the coefficients of the Weil divisor $L$ on $X_0$, and write $L = \tilde m L_{\rm red}$. Then one of the following holds. Next three figures show that modification process from $L_{\rm red}$ in $\bar S$ to $L_{\rm red}$ in $X_0$. Figure~\ref{$m{I}_d(r, a)$} is for the type $m{I}_d(r, a)$, Figure~\ref{$II(r)$} is for the type $II(r)$, and Figure~\ref{$III, IV$} is for the type $III(2), III(3), IV(2)$.

\begin{enumerate}
\item type $m{I}_d(r, a)$, $\tilde m=mr$; \,\, $L_{\rm red}$ is a reduced cycle of $d$ nonsingular rational curves (resp. a reduced rational curve with one node) if $d > 1$ (resp. $d = 1$), and ${\rm Sing}(X_0)\cap {\rm Supp}(L)$ consists of quotient singularities of type $1/r^2(a, r-a)$ at $d$ double points of $L_{\rm red}$.

\begin{figure}[hbtb]
\begin{center}
\includegraphics[scale=0.23]{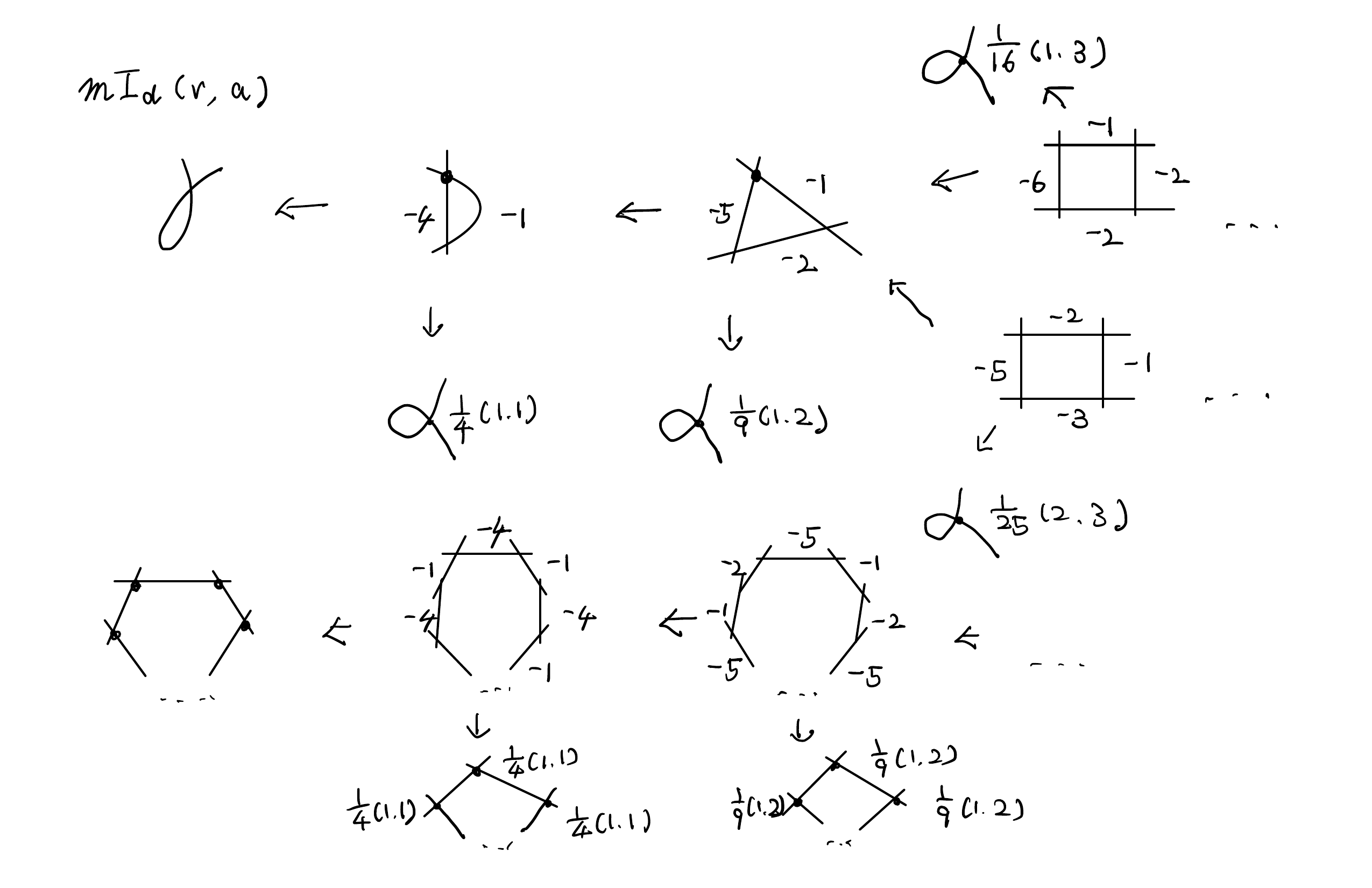}
\end{center}
\caption{Modification for $m{I}_d(r, a)$}
\label{$m{I}_d(r, a)$}
\end{figure}

\item type $II(r)$, $r=2$ or $3$, $\tilde m =r$;\,\,  $L_{\rm red}$ is a rational curve with one cusp, and ${\rm Sing}(X_0)\cap {\rm Supp}(L)$ consists of a quotient singularity of type $1/r^2(1, r-1)$ at the cusp of $L_{\rm red}$.

\item type $II(r)$, $r=4$ or $5$, $\tilde m =rr'$, $r' =2$ (resp. 5) if $r =4$ (resp. 5); \,\, $L_{\rm red}$ is a nonsingular rational curve, and  ${\rm Sing}(X_0)\cap {\rm Supp}(L)$ consists of two quotient singularities of types $1/r^2(3, r - 3)$ and $1/r'^2(1, r' -1)$.

\begin{figure}[hbtb]
\begin{center}
\includegraphics[scale=0.23]{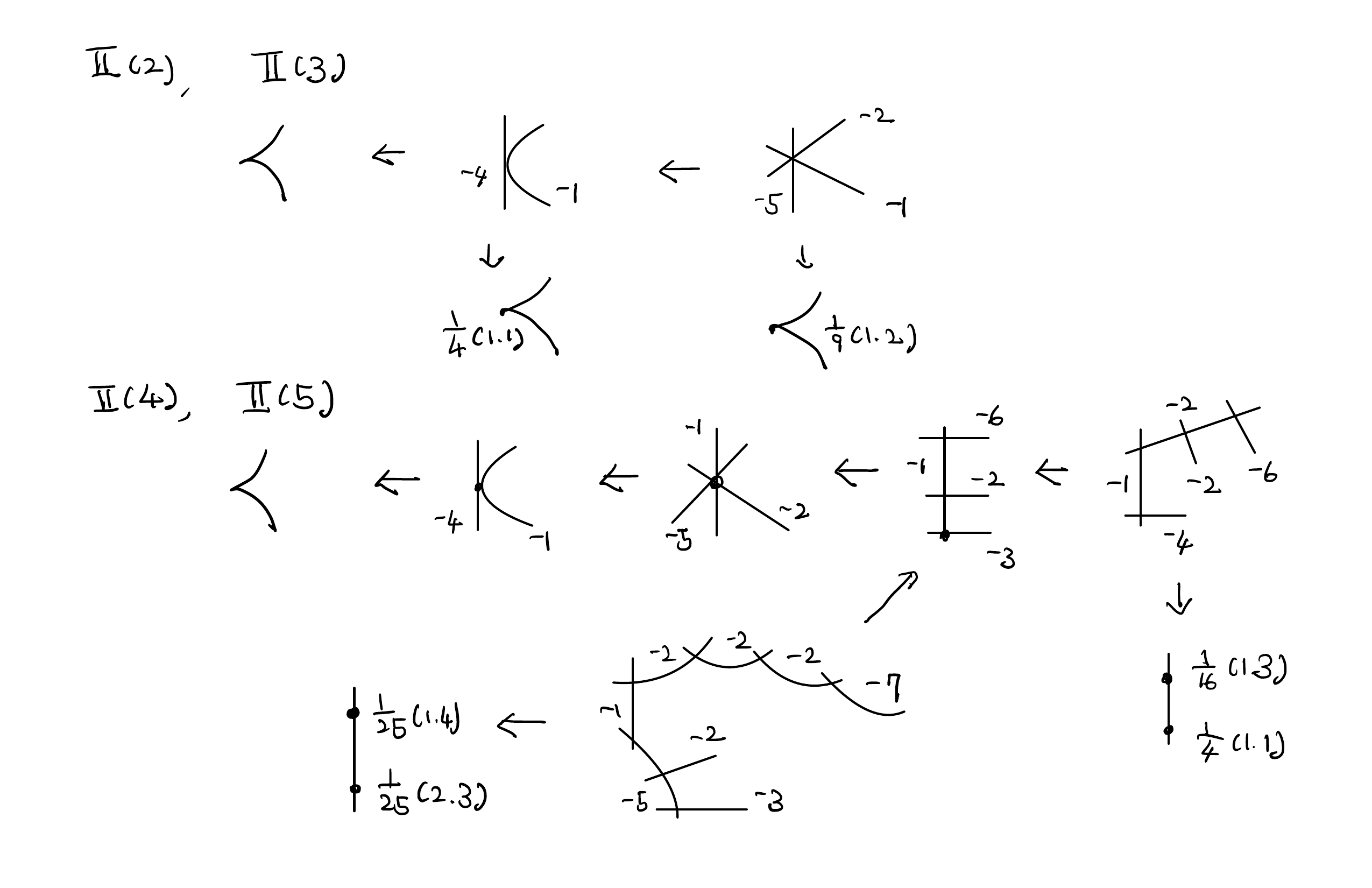}
\end{center}
\caption{Modification for $II(r)$}
\label{$II(r)$}
\end{figure}

\item type $III(2)$, $\tilde m = 2, r=2$; \,\, $L_{\rm red} = F_1 +2F_2$, where the $F_j$ are nonsingular rational curves intersecting transversally at a point, and ${\rm Sing}(X_0)\cap {\rm Supp}(L)$ consists of two quotient singularities of type $1/4(1,1)$ on nonsingular points of ${\rm Supp}(L)$ on $F_2$.

\item type $III(3)$, $\tilde m = 9, r=3$; \,\, $L_{\rm red} = F_1 +F_2$, where the $F_j$ are nonsingular rational curves intersecting transversally at a point, and ${\rm Sing}(X_0)\cap {\rm Supp}(L)$ consists of three quotient singularities of type $1/9(1,2)$ such that one of them is at $F_1\cap F_2$ and the other two are at nonsingular points of $L_{\rm red}$ one on each $F_j$. 

\item type $IV(2)$, $\tilde m=4, r=2$; \,\,  $L_{\rm red} = F_1 +F_2 +F_3$, where the $F_j$ are nonsingular rational curves intersecting transversally at one point, and ${\rm Sing}(X_0)\cap {\rm Supp}(L)$  consists of four quotient singularities of type $1/4(1,1)$ such that one of them is at $F_1 \cap F_2 \cap F_3$ and the other three are at nonsingular points of $L_{\rm red}$ one on each $F_j$.

\begin{figure}[hbtb]
\begin{center}
\includegraphics[scale=0.23]{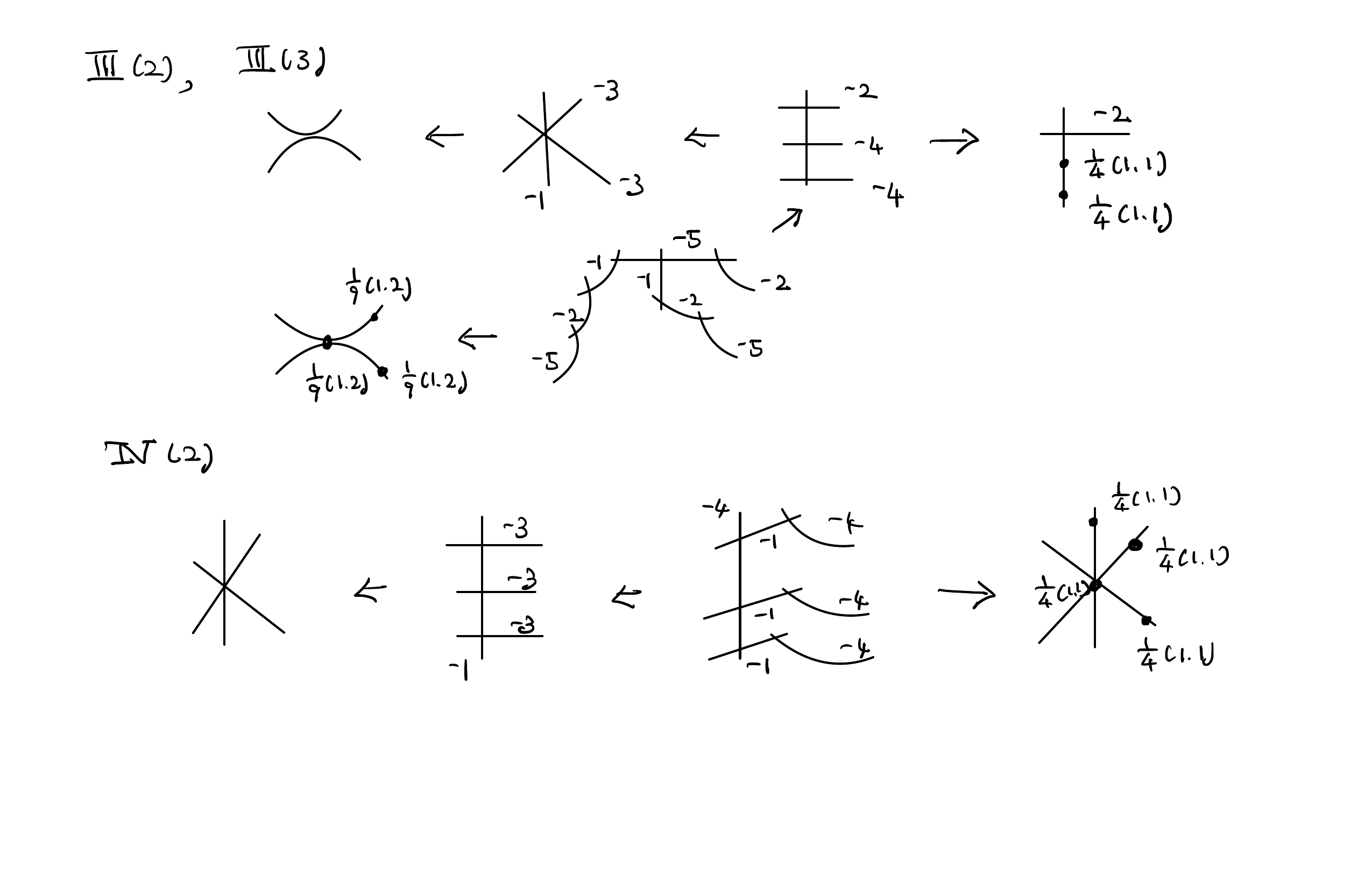}
\end{center}
\caption{Modification for $III$ and $IV$}
\label{$III, IV$}
\end{figure}
\end{enumerate}

Singular fiber types $mI_0, I_b^*, II^*, III^*, IV^*$ do not occur on the central fiber of $f$ when $f$ is not smooth.

And the canonical bundle formula holds:
$$K_{X_0}\sim_{\QQ} g_{X_0}^*{\bf d}+\sum_k(m^{(k)}-1)F^{(k)},$$
where ${\bf d}$ is some divisor (moduli divisor) on $B_0$ and the summation is taken for all the multiple fibers $L^{(k)}=m^{(k)}F^{(k)}$ whose multiplicities $m^{(k)}$ are defined by the following table.
\[\begin{array}{ccccc}
\text{type of $L^{(k)}$} & mI_d(r, a)  & II(r) & III(r) & IV(r) \\
m^{(k)} & mr  & r & r& r
\end{array}\]

Consider a relatively minimal smooth rational elliptic surface $Y$ with a section. Assume that $Y$ has at least two singular fibers of type $I_n, II, III$, or $IV$. Then one can construct a normal rational elliptic surface $X$ by replacing these two singular fibers with the singular fibers in the above theorem, by using blow-ups and Artin's contractibility theorem \cite[Theorem~2.3]{Art}: the above three figures explain that modification process from $L_{\rm red}$ in $\bar S=Y$ to $L_{\rm red}$ in $X_0=X$.

\begin{Thm}\label{smoothing}
Let $X$ be a normal projective surface with an elliptic fibration. Suppose that the minimal resolution of $X$ is a rational elliptic surface with a section and non-isotrivial elliptic fibration. Suppose there are at most two singular fibers $L$ whose $L_{\rm red}$ is one of the types $I_d(r, a), II(r), III(2)$, $III(3)$, or $IV(2)$ in the above. Then there exists a $\QQ$-Gorenstein smoothing to elliptic surfaces. 

If there is exactly one such singular fiber $L_{\rm red}$, this $\QQ$-Gorenstein smoothing gives a rational elliptic surface of index $r$. If there are two such singular fibers $L_{\rm red}$, their $\QQ$-Gorenstein smoothing gives an elliptic surface with non-negative Kodaira dimension, unless both multiplicities are two, in which case it gives an Enriques surface.
\end{Thm}
\begin{proof}
By the method in \cite[Section~2]{LP}, or \cite[Section~6]{LN}, or \cite[Section~2]{CL}, we obtain a $\QQ$-Gorenstein smoothing $f: \cX\to\De$ such that $f^{-1}(0)=X$ and a general fiber $X_t$ is a relatively minimal smooth elliptic surface with $p_g=q=0$. All singularities appeared in singular fibers $L_{\rm red}$ in $X$ are special cyclic singularities of the type $\frac{1}{r^2}(r, r-a)$ (it is called a $T$-class singularity \cite{KSB}). 

This singularity has a local $\QQ$-Gorenstein smoothing. Let $\pi: Y\to X$ be the minimal resolution of $X$ with reduced exceptional divisor $E$. By \cite[Lemma 1 and Theorem 2]{LP}, it is enough to show that $H^2(Y, T_Y(-{\rm log}(E)))=0$ to prove no obstruction from local to global $\QQ$-Gorenstein smoothing. At most two singular fibers with $T$-class singularities are crucial to show $H^2(Y, T_Y(-{\rm log}(E)))=0$ by the same proof of \cite[Lemma 2]{LP} or \cite[Lemma 6.5]{LN}. By our assumption, we also note that $Y$ is a rational elliptic surface with non-isotrivial elliptic fibration. 

By the $\QQ$-Gorenstein condition, the canonical bundle formula for $X_t$ takes the form
$$K_{X_t}\sim  g_{X_t}^*{\cO_{B_t}}(-1)+(m^{(1)}-1)F^{(1)}+(m^{(2)}-1)F^{(2)}$$
where two multiple fibers are $L^{(k)}=m^{(k)}F^{(k)}$, and the multiplicities $m^{(k)}$ are defined by the above table. Hence $K_{X_t}$ is $\QQ$-effective unless $(m^{(1)},m^{(2)})=(2,2)$.
\end{proof}

\begin{Rmk}
The advantage of constructing Dolgachev surfaces with $\QQ$-Gorenstein smoothings is that it is extendable to arbitrary 
characteristic. For instance, consider a rational elliptic surface obtained by blowing up the base points of the following pencil of cubics:
\[ \phi_0=xyz, \quad \phi_\infty=(x+y)(y+z)(x+y+z).\]
This rational elliptic surface has two $I_5$ singular fibers (cf. \cite[Section~4.9]{CDL}) and can be defined over ${\rm Spec}\, \ZZ$. One can produce an elliptic surface having two multiple fibers with arbitrary multiplicities over an algebraically closed field of any 
characteristic $p\ge 0$ via $\QQ$-Gorenstein smoothings \cite[Section~6]{LN}.
\end{Rmk}

\section{Limit surface when a multiple fiber approaches to an additive type fiber}\label{limit}

\medskip
\begin{Lem}\label{lct}
Let $L_0=L_{\rm red}$ in $X_0$  in Theorem~\ref{classification}. Then the log canonical threshold (lct) of $(X_0, L_0)$  for each additive type is the following.
\begin{itemize}
\item ${\rm lct}(X_0, L_0)=2/3$ for type $II(2)$,
\item ${\rm lct}(X_0, L_0)=1/2$ for type $II(3)$,
\item ${\rm lct}(X_0, L_0)=2/3$ for type $II(4)$,
\item ${\rm lct}(X_0, L_0)=5/6$ for type $II(5)$,
\item ${\rm lct}(X_0, L_0)=1$ for type $III(2)$,
\item ${\rm lct}(X_0, L_0)=3/4$ for type $III(3)$,
\item ${\rm lct}(X_0, L_0)=2/3$ for type $IV(2)$.
\end{itemize}
\end{Lem}
\begin{proof}
Let $p: Y_0\to X_0$ be the log resolution of $X_0$ given in  Theorem~\ref{classification}. Set $X:=X_0$. We have the following configuration of exceptional curves and the proper transform of $L_0$ in $Y_0$.

\begin{figure}[hbtb]
\begin{center}
\includegraphics[scale=0.2]{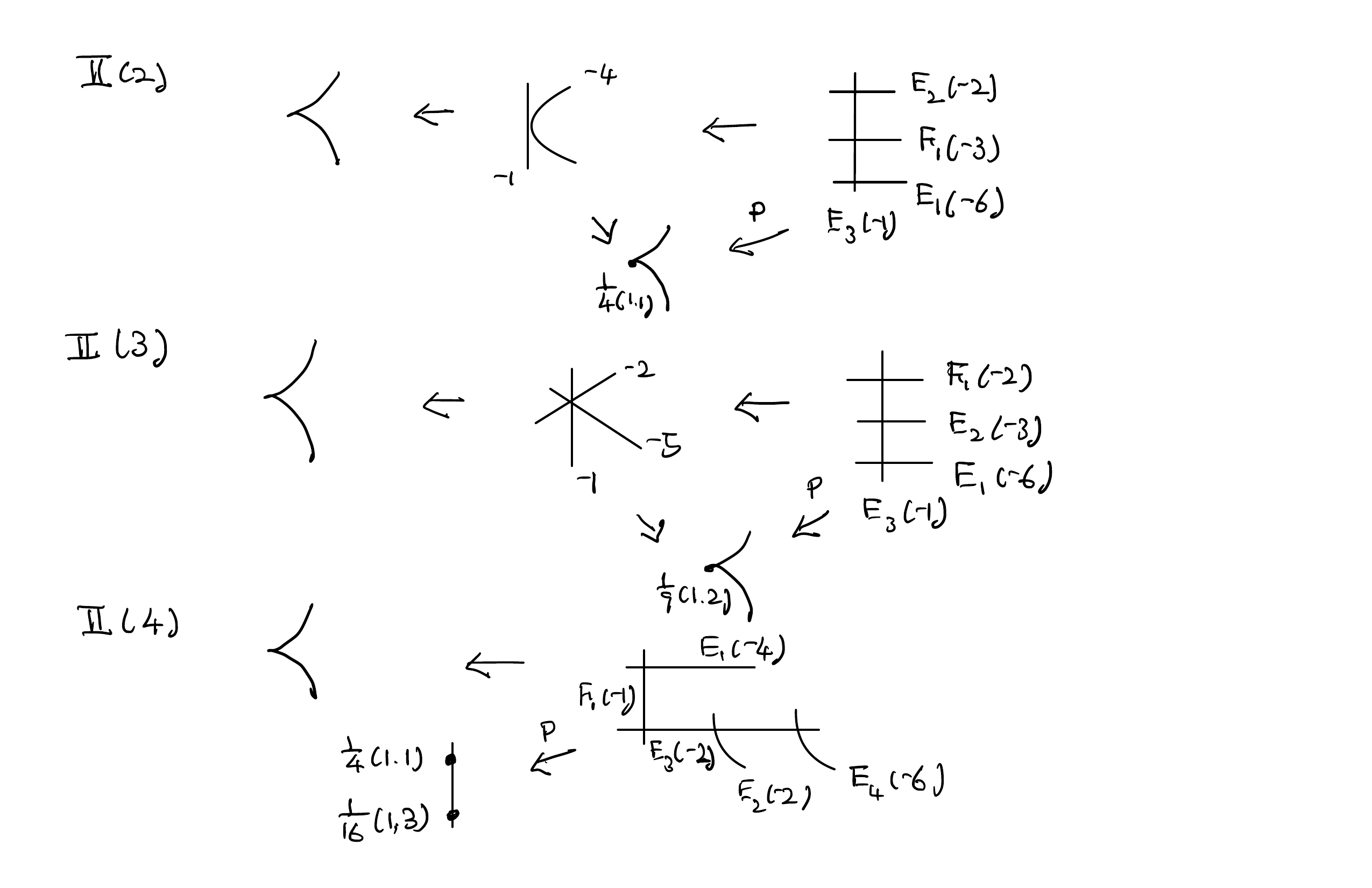}
\includegraphics[scale=0.2]{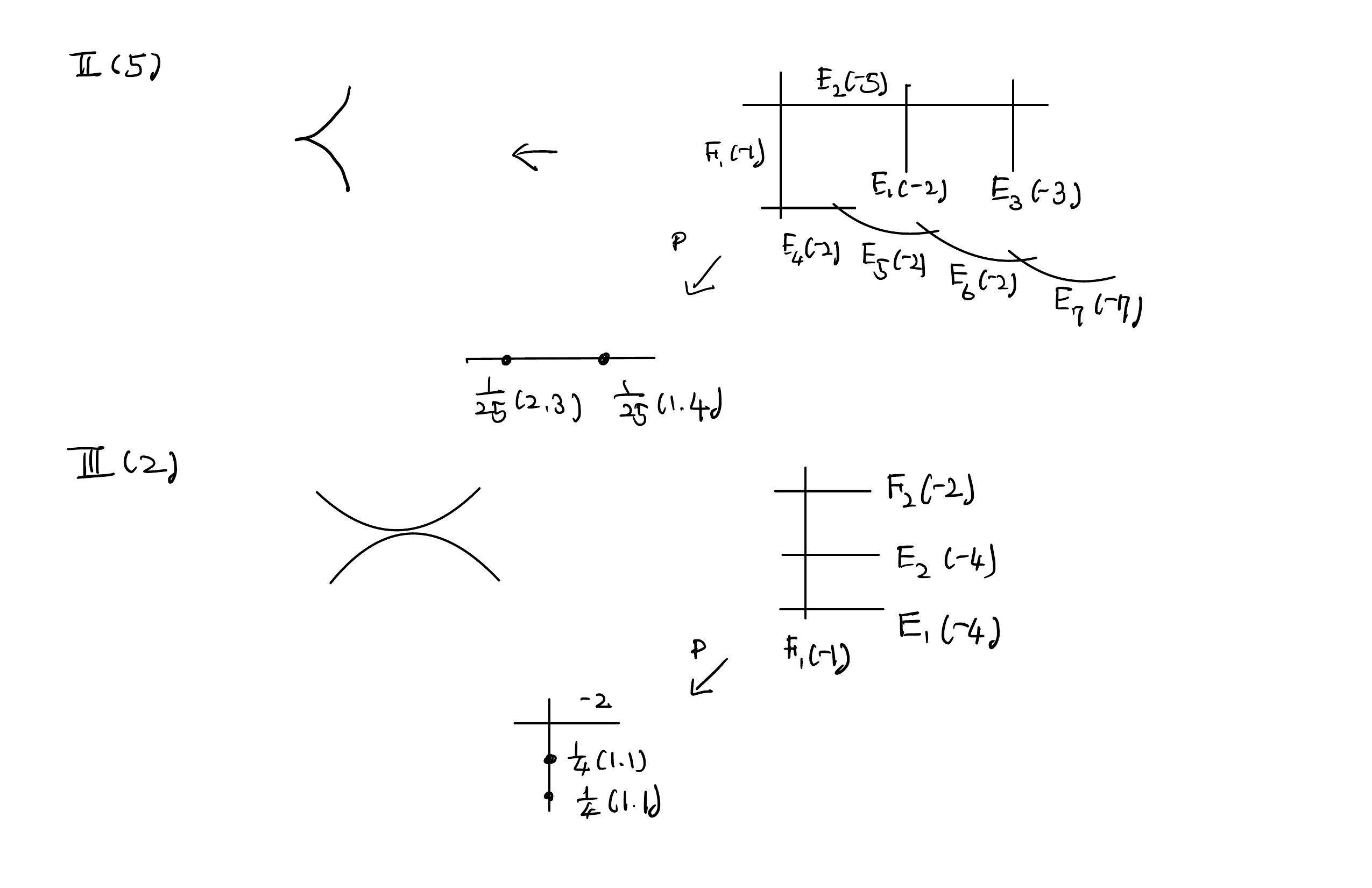}
\includegraphics[scale=0.2]{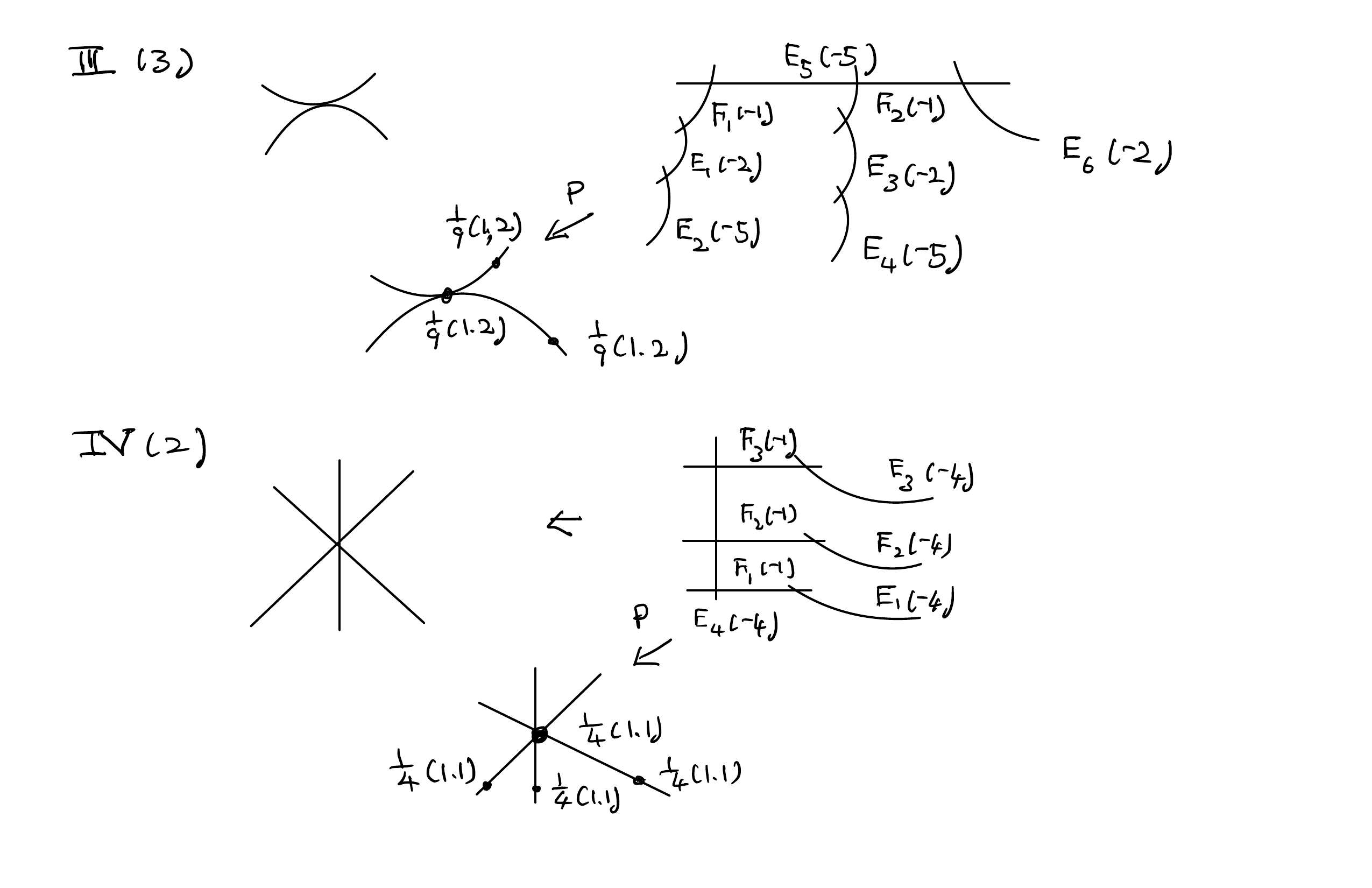}
\end{center}
\caption{Log resolution of $X_0$}
\label{fig-lct}
\end{figure}

Then we obtain the followings by an easy computation. Let us compute the case $II(2)$. All other cases are computed by the same way. 
Set $K_{Y_0}=p^*K_X+a_1E_1+a_2E_2+a_3E_3+a_4F_1$, and $p^*L_0=b_1E_1+b_2E_2+b_3E_3+b_4F_1$. Then we can compute $a_i$ and $b_i$ by using the self intersection number of $E_i$ and $F$, their intersection numbers, and the adjunction formula ($K_{Y_0}E_i+E_i^2=-2$). Then we get $a_1=-\frac{1}{2}, a_2=\frac{1}{2}, a_3=1, a_4=0$. And $b_1=\frac{1}{2}, b_2=\frac{3}{2}, b_3=3, b_4=1$.

\begin{itemize}
\item[$II(2)$:] $K_{Y_0}=p^*K_X-1/2E_1+1/2E_2+E_3$, \\ $p^*L_0=1/2E_1+3/2E_2+3E_3+F_1$, \smallskip
\item[$II(3)$:] $K_{Y_0}=p^*K_X-2/3E_1-1/3E_2$, \\ $p^*L_0=1/3E_1+2/3E_2+2E_3+F_1$, \smallskip
\item[$II(4)$:] $K_{Y_0}=p^*K_X-1/2E_1-1/4E_2-1/2E_3-3/4E_4$, \\$p^*L_0=1/4E_1+3/8E_2+3/4E_3+1/8E_4+F_1$,  \smallskip
\item[$II(5)$:] $K_{Y_0}=p^*K_X-2/5E_1-4/5E_2-3/5E_3-1/5E_4-2/5E_5-3/5E_6-4/5E_7$, \\ $p^*L_0=3/25E_1+6/25E_2+2/25E_3+19/25E_4+13/25E_5+7/25E_6+1/25E_7+F_1$,  \smallskip
\item[$III(2)$:] $K_{Y_0}=p^*K_X-1/2E_1-1/2E_2$, \\ $p^*L_0=1/4E_1+1/4E_2+F_1+F_2$, \smallskip
\item[$III(3)$:] $K_{Y_0}=p^*K_X-1/3E_1-2/3E_2-1/3E_3-2/3E_4-2/3E_5-1/3E_6$, \\ $p^*L_0=5/9E_1+1/9E_2+5/9E_3+1/9E_4+4/9E_5+2/9E_6+F_1+F_2$,  \smallskip
\item[$IV(2)$:] $K_{Y_0}=p^*K_X-1/2E_1-1/2E_2-1/2E_3-1/2E_4$, \\ $p^*L_0=1/4E_1+1/4E_2+1/4E_3+3/4E_4+F_1+F_2+F_3$. \smallskip
\end{itemize}
Moreover, ${\rm lct}(X_0, L_0)={\rm min}\{\frac{b_j+1}{r_j}, 1\}$ if $K_{Y_0/X_0}=\sum b_i\bar E_i$ and $p^*L_0=\sum r_i\bar E_i$ for prime divisors $\bar E_i$ on $Y_0$ (cf. \cite{Laz}, Example 9.3.16).
\end{proof}

The notion of twisted fiber of a rational elliptic surface with a section is introduced by Ascher and Bejleri \cite[Definition 3.4]{AB21}. In Theorem~\ref{gluing}, twisted fiber means that a gluing curve $E$ is irreducible and $(X, E_{\rm red})$ has (semi-)log canonical singularities. 

\begin{Thm}\label{gluing}
The following slt surfaces $X$ have a $\QQ$-Gorenstein smoothing to rational elliptic surfaces of index 1.
\begin{enumerate}
\item $X=X_1\cup_{\PP^1} X_2$ is the union of two rational elliptic surfaces of index 1 with RDPs glued along twisted $I_0^*$ (the double curve is a $\PP^1$ and four $A_1$ singularities of $X_i$ are placed on this $\PP^1$, and locally around each point the surface is a quotient of a nodal surface by $\ZZ/2\ZZ$). 
\item $X=X_1\cup_{\PP^1} X_2$ is the union of two rational elliptic surfaces of index 1 with cyclic quotient singularities along twisted $II\, |\, II^*$ (the double curve is a $\PP^1$, $\frac{1}{2}(1, 1)$, $\frac{1}{3}(1,1$), $\frac{1}{6}(1,1)$ singularities of $X_1$ are placed on this $\PP^1$, $\frac{1}{2}(1, 1), \frac{1}{3}(1,2), \frac{1}{6}(1,5)$ singularities of $X_2$ are placed respectively on this $\PP^1$, and locally around each point of the surface is a quotient of a nodal surface by $\ZZ/2\ZZ$, $\ZZ/3\ZZ$, $\ZZ/6\ZZ$, respectively). 
\item $X=X_1\cup_{\PP^1} X_2$ is the union of two rational elliptic surfaces of index 1 with cyclic quotient singularities along twisted $III\, |\, III^*$ (the double curve is a $\PP^1$, $\frac{1}{2}(1, 1), \frac{1}{4}(1,1), \frac{1}{4}(1,1)$ singularities of $X_1$ are placed on this $\PP^1$, $\frac{1}{2}(1, 1)$, $\frac{1}{4}(1,3)$, $\frac{1}{4}(1,3)$ singularities of $X_2$ are placed respectively on this $\PP^1$, and locally around each point of the surface is a quotient of a nodal surface by $\ZZ/2\ZZ$, $\ZZ/4\ZZ$, $\ZZ/4\ZZ$, respectively). 
\item $X=X_1\cup_{\PP^1} X_2$ is the union of two rational elliptic surfaces of index 1 with cyclic quotient singularities along twisted $IV\, |\, IV^*$ (the double curve is a $\PP^1$, three $\frac{1}{3}(1, 1)$ singularities of $X_1$ are placed on this $\PP^1$, three $A_2$ singularities of $X_2$ are placed respectively on this $\PP^1$, and locally around each point of the surface is a quotient of a nodal surface by $\ZZ/3\ZZ$). 
\end{enumerate}
\end{Thm}
\begin{proof} $\QQ$-Gorenstein deformation theory is a special deformation theory which satisfies $\QQ$-Gorenstein condition. We consider the sheaves $\cT^i_{\QQ G,X}$ for $i=0,1,2$ as in \cite[Section 3]{Hac} that relate $\QQ$-Gorenstein deformations of $X$. We need to show that $H^j(X,\cT^i_{\QQ G, X})=0$ for $i+ j=2$ to show no obstruction for a $\QQ$-Gorenstein smoothing of $X$. 

(1) is proved in \cite[Proposition~4.11]{AB22}. Others can be proven in a similar manner. In each case, $X$ has local canonical covering by a local complete intersection, so that $H^0(\cT^2_{\QQ G, X})=0$. The sheaf $\cT^1_{\QQ G,X}$ is supported on the singular locus of $X$ which lie on the gluing $\PP^1$. Let $i: \PP^1\hookrightarrow X$. By \cite[Proposition~3.6]{Has}, $\cT^1_{\QQ G, X}=i_*\cO_{\PP^1}(n)$ where $n=0$ in all the cases because $\mathrm{Diff}=\frac{m-1}{m}$ where $(\PP^1, X_i)\cong ((x=0)\subset \CC^2/{\ZZ}_m)$ and ${\ZZ}_m$ acts with weights $(1, q)$ with $(q, m)=1$ (cf. \cite[Proposition~16.6]{Cor}). For instance, $n=2(\frac{1}{2}+\frac{2}{3}+\frac{5}{6}+{\rm deg}K_{\PP^1})=0$. So we have $H^1(\cT^1_{\QQ G, X})=0$.

Let $(X_i, E_i)$ for $i = 1,2$ denote the two components and $E_i = E|_{X_i}$ denote the restriction of the double locus. Following \cite[Lemma~9.4]{Hac}, to show that $H^2(T_X) = 0$, it suffices to show
that $H^2(T_{X_i} (-E_i)) = 0$. We note that $H^2(T_X) = 0$ implies that $H^2(X, \cT^0_{\QQ G, X})=0$. This is equivalent to showing that $\cO_{X_i}(-K_{X_i}-E_i)$ has a non-zero section, since every such section induces an injection $H^0(\Omega_{X_i}\otimes\cO(K_{X_i}+E_i))\to H^0(\Omega_{X_i})=0$. 
Let $f_i: X_i\to\PP^1$ be an elliptic fibration and $F_i$ be a general fiber of $f_i$ for $i=1, 2$. We note that $K_{X_1}=-F_1+B$ where $B$ is effective and $B=(r-1)E_1$ by \cite[Theorem~1.2]{AB17}. Here $r=1$ for (1), $r=5$ for (2), $r=3$ for (3), $r=2$ for (4), respectively. So we have $-K_{X_1}-E_1=F_1-rE_1=E_1$ because $F_1=(r+1)E_1$. And since $X_2$ is obtained by contracting $(-2)$-curves, $K_{X_2}=-F_2$. Therefore $-K_{X_2}-E_2=F_2-E_2= rE_2$ because $F_2=(r+1)E_2$. So the reflexive sheaf $-K_{X_i}-E_i$ has a non-zero section for $i=1, 2$. 

Therefore, $X$ has a $\QQ$-Gorenstein smoothing $\cX\to\De$ and $(r+1)K_{\cX}$ is Cartier. Since we have 
$-(K_{X_1}+E_1)=E|_{X_1}$ and $-(K_{X_2}+E_2)=rE|_{X_2}$, 
$$K_X\sim_{\QQ} -(r+1)E \sim-{\rm fiber}.$$ 
Therefore $X$ is deformed to a rational elliptic surface without a multiple fiber.
\end{proof}

\begin{Cor}\label{gluing2}
Let $X=X_1\cup_{\PP^1}X_2$ be the slt union of two rational elliptic surfaces of indices $m_1$ and $m_2$ with $(m_1,m_2)=1$, whose gluing is one of the types in Theorem~\ref{gluing}.
Let $m_iF_i$ be the multiple fiber of $X_i$, where $F_i$ denotes its reduced support.
Assume that each $F_i$ is smooth.
Let $Y=J(X_1)\cup_{\PP^1}J(X_2)$,
where $J(X_i)$ denotes the Jacobian surface of $X_i$.
Let $b_i\in \PP^1\cup\PP^1$ be the point over which $F_i$ lies.
Assume that the moduli map
\begin{equation}\label{modulimap}
	\PP^1\cup\PP^1 \dashrightarrow \mathcal M_{1,1}
\end{equation}
associated to the elliptic fibration $Y\to \PP^1\cup\PP^1$ is \'etale at $b_i$.
Then $X$ admits a $\QQ$-Gorenstein smoothing to Dolgachev surfaces of type $(m_1,m_2)$.
\end{Cor}
\begin{proof}
	By Theorem~\ref{gluing}, $Y$ has a $\QQ$-Gorenstein smoothing $\mathcal Y\to \mathcal B\to \Delta$ to rational elliptic surfaces of index 1. 
	Let $\mathcal U\subset \mathcal B$ be the open subset over which the fibers of $\mathcal Y\to \mathcal B$ are  smooth elliptic curves. 
	Then $b_i\in \mathcal U$, and we have the moduli map 
	$\varphi:\mathcal U\to \mathcal M_{1,1}$.
	
	Since $\mathcal B$ is smooth at $b_i$, so is $\mathcal U$.
	By assumption, the differential of $\varphi$ at $b_i$ is nonzero. Hence $\varphi$ is smooth at $b_i$. 
	Let $C_i\subset \mathcal U$ be the local branch of $\varphi^{-1}(\varphi(b_i))$ through $b_i$. 
	By shrinking $\Delta$ if necessary, we may assume that $C_i\subset \varphi^{-1}(\varphi(b_i))$ is an irreducible component, and that the map $C_i\to \Delta$ is an isomorphism, since $\varphi$ is smooth at $b_i$. Let $\mathcal F_i:=\mathcal Y|_{C_i}$. By construction, $\mathcal F_i\to \Delta$ is an \'etale locally trivial deformation of $F_i$.
    
	Since $F_i$ are smooth and $(\mathcal F_i)_t\cong F_i$ for $t\in\Delta$, we can perform the generalized logarithmic transforms along $(\mathcal F_i)_t$ in the family \cite[p.~8]{Fuj}. Performing the generalized logarithmic transforms along these two $(\mathcal F_i)_t$ to make the multiple fibers $m_i(\mathcal F_i)_t$ produces Dolgachev surfaces $X_t$ of type $(m_1, m_2)$. This family provides a $\QQ$-Gorenstein smoothing of $X$.
\end{proof}

\begin{Rmk}
	Let $\sigma\in {\bf T}^1_{\QQ G, Y}$ be the nonzero element associated to $\mathcal Y\to \Delta$. The above construction of $\mathcal F_i\to \Delta$ determines an element $\tilde \sigma \in {\bf T}^1_{\QQ G, (Y,F_1+F_2)}$. The element $\tilde \sigma$ can also be determined as follows.
	Consider the exact sequence associated to the ($\QQ$-Gorenstein) deformation theory of the pair $(Y, F_1+F_2)$ 
	\cite[Proposition 3.1, Corollary 3.2]{Has}:
$$H^0(\cO_{F_1}(F_1)\oplus \cO_{F_2}(F_2)) \xrightarrow{~i~} {\bf T}^1_{\QQ G, (Y,F_1+F_2)} \xrightarrow{~\rho~} {\bf T}^1_{\QQ G, Y}. $$ 
Then $\sigma$ lies in the image of $\rho$.
Indeed, since $H^2(Y,\cO_Y)=0$, the Cartier divisor $F_1+F_2\subset Y$ extends to an effective relative Cartier divisor along the first-order deformation of $Y$.
This defines an element of ${\bf T}^1_{\QQ G,(Y,F_1+F_2)}$ mapping to $\sigma$.

We next show that $\sigma$ admits a lift in ${\bf T}^1_{\QQ G, (Y,F_1+F_2)}$ whose induced first-order deformations of $F_1$ and $F_2$ are trivial.
Let $j:{\bf T}^1_{\QQ G, (Y,F_1+F_2)} \to H^1(F_1,T_{F_1})\oplus H^1(F_2,T_{F_2})$ be the map induced by the Kodaira--Spencer maps associated to the deformations of $F_i$.
Then
\[j\circ i:~ \CC^2\cong H^0(\cO_{F_1}(F_1)\oplus \cO_{F_2}(F_2))\longrightarrow H^1(F_1,T_{F_1})\oplus H^1(F_2,T_{F_2})\cong \CC^2\]
is the Kodaira--Spencer map associated to the fibers $F_1\sqcup F_2\subset Y$.
Since $H^0(\cO_{F_i}(F_i))\to H^1(T_{F_i})$ can be identified with the differential of the moduli map \eqref{modulimap} at $b_i$, it is an isomorphism by the assumption. Hence, the map $j\circ i$ is an isomorphism, and in particular, it is surjective. 
Therefore, there exists a unique $\tilde \sigma \in \rho^{-1}(\sigma)$ such that $j(\tilde \sigma)=0$.
\end{Rmk}

\begin{Rmk} 
After the first version of our paper appeared on arXiv, Canedo and Urz\'ua \cite{CU} constructs $\QQ$-Gorenstein deformations whose central fiber has a non lc ${\rm\QQ}$HD singularity and a general fiber is a Dolgachev surface. For each such degeneration, they construct our model via a Seifert partial resolution in the sense of Wahl followed by semistable flips. 
\end{Rmk}

By Theorems~\ref{Mir80-1} and~\ref{Mir80-2}, in Miranda's GIT compactification a polystable generically smooth cubic pencil has only singular fibers of types $I_n$, $II$, $III$, $IV$, or two fibers of type $I_0^*$. 
Thus, in studying degenerations of elliptic surfaces with multiple fibers obtained from such rational elliptic surfaces by logarithmic transformations along smooth fibers, it is natural to ask what happens when a multiple fiber collides with one of the above singular fibers.
Theorems~\ref{smoothing} and~\ref{gluing} and Corollary~\ref{gluing2} provide partial answers to this question. 

Moreover, when the multiplicity of the multiple fiber is at most $5$, Lemma~\ref{lct} suggests wall-crossing phenomena as one varies the coefficient $c$ in pairs of the form $(X,cF)$, where $F$ denotes the reduced multiple fiber.


\begin{thebibliography}{12}

\bibitem{Art} 
M.~Artin, 
\textit{Some numerical criteria for contractability of curves on algebraic surfaces},
Amer. J. Math. \textbf{84} (1962), 485--496.

\bibitem{AB17} 
K.~Ascher and D.~Bejleri, 
\textit{Log canonical models of elliptic surface},
Adv. Math. \textbf{320} (2017), 210--243.

\bibitem{AB21} 
K.~Ascher and D.~Bejleri, 
\textit{Moduli of weighted stable elliptic surfaces and invariance of log plurigenera. With an appendix by Giovanni Inchiostro},
Proc. Lond. Math. Soc. (3) \textbf{122} (2021), no. 5, 617--677.

\bibitem{AB22} 
K.~Ascher and D.~Bejleri,  
\textit{Compact moduli of degree one del Pezzo surfaces}, 
Math. Ann. \textbf{384} (2022), no. 1-2, 881--911.

\bibitem{BHPV} W.~Barth, K.~Hulek, C.~Peters, and A.~Van de Ven, 
Compact complex surfaces, 
Second edition. Ergebnisse der Mathematik und ihrer Grenzgebiete. 3. Folge. Springer-Verlag, Berlin, 2004.


\bibitem{Bir} 
C.~Birkar,
\textit{Moduli of algebraic varieties},
arXiv:2211.11237.


\bibitem{CU}
M.~Canedo and G.~Urz\'ua,
\textit{Rational homology disk degenerations of elliptic surfaces},
arXiv:2605.06668.


\bibitem{CD12}
S.~Cantat and I.~Dolgachev,
\textit{Rational surfaces with a large group of automorphisms},
J. Amer. Math. Soc. \textbf{25} (3) (2012), 863--905.

\bibitem{CL} 
Y.~Cho and Y.~Lee, 
\textit{Exceptional collections on Dolgachev surfaces associated with degenerations}, 
Adv. Math. \textbf{324} (2018), 394--436.

\bibitem{Cor}
A.~Corti, 
\textit{Adjunction of log divisors}.
Flips and abundance for algebraic threefolds.
Papers from the Second Summer Seminar on Algebraic Geometry held at the University of Utah, Salt Lake City, Utah, August 1991. Ast\'erisque No. 211 (1992). 

\bibitem{CDL}
F.~Cossec, I.~Dolgachev, and C.~Liedtke,
Enriques Surfaces I, 
Second edition. 
Springer, Singapore, 2025. 



\bibitem{Dol}
I.~Dolgachev, 
\textit{Algebraic surfaces with $p_g = q = 0$},
in Algebraic Surfaces, CIME 1977, Liguori Napoli, 1981, 97--215.


\bibitem{FM88}
R.~Friedman and J.~Morgan,
\textit{On the diffeomorphism types of certain algebraic surfaces. I}, 
J. Differential Geom. \textbf{27} (1988), no. 2, 297--369. 


\bibitem{Fuj} Y.~Fujimoto, 
\textit{On rational elliptic surfaces with multiple fibers},
Publ. Res. Inst. Math. Sci. \textbf{26} (1990), no. 1, 1--13.

\bibitem{Hac}
P.~Hacking,
\textit{Compact moduli of plane curves},
Duke Math. J. \textbf{124} (2) (2004), 213--257.

\bibitem{Har} 
R.~Hartshorne,
Algebraic geometry. Graduate Texts in Mathematics, No. 52 (1977), Springer-Verlag, New York-Heidelberg.
 
\bibitem{Has}
B.~Hassett,
\textit{Stable log surfaces and limits of quartic plane curves}, 
Manuscripta Math. \textbf{100} (1999), no. 4, 469--487.

\bibitem{Kaw}
Y.~Kawamata, 
\textit{Moderate degenerations of algebraic surfaces},
Complex algebraic varieties (Bayreuth, 1990), 113–132, Lecture Notes in Math., 1507, Springer, Berlin, 1992.

\bibitem{Kol23}
J.~Koll\'ar, 
\textit{Families of varieties of general type}, 
Cambridge Tracts in Mathematics, vol.~231, 
Cambridge University Press, Cambridge, 2023. 
With the collaboration of K.~Altmann and S.~J.~Kov\'acs.

\bibitem{KM}
J.~Koll\'ar and S.~Mori, 
Birational geometry of algebraic varieties. With the collaboration of C. H. Clemens and A. Corti,
Cambridge Tracts in Mathematics, 134. Cambridge University Press, Cambridge. 

\bibitem{KSB}
J.~Koll\'ar and N.~I.~Shepherd-Barron, 
\textit{Threefolds and deformations of surface singularities},
Invent. Math. \textbf{91} (1988), no. 2, 299--338.

\bibitem{Laz}
R.~Lazarsfeld, Positivity in algebraic geometry. II. Positivity for vector bundles, and multiplier ideals. Ergebnisse der Mathematik und ihrer Grenzgebiete. 3. Folge. 
A Series of Modern Surveys in Mathematics, 49. Springer-Verlag, Berlin, 2004.

\bibitem{LN}
Y.~Lee and N.~Nakayama, 
\textit{Simply connected surfaces of general type in positive characteristic via deformation theory}, 
Proc. Lond. Math. Soc. (3) \textbf{106} (2013), no. 2, 225--286. 

\bibitem{LP}
Y.~Lee and J.~Park, 
\textit{A simply connected surface of general type with $p_g = 0$ and $K^2 = 2$},
Invent. Math. \textbf{170} (2007), 483--505.


\bibitem{Mir80}
R.~Miranda, 
 \textit{On the stability of pencils of cubic curves},
  Amer. J. Math.\textbf{102} (1980), no. 6, 1177--1202.
  
  
\bibitem{Mir81}
R.~Miranda, 
 \textit{The moduli of Weierstrass fibrations over $\PP^1$}, 
 Math. Ann. \textbf{255} (1981), no. 3, 379--394.
 


\bibitem{Zar}
A.~Zanardini, 
\textit{Stability of pencils of plane sextics and Halphen pencils of index two},
 Manuscripta Math. \textbf{172} (2023), no. 1-2, 353–374.


   
\end{thebibliography}
\end{document}